\documentclass{daj}

\dajAUTHORdetails{%
  title = {Convolutions of sets with bounded VC-dimension are uniformly continuous}, 
  author = {Olof Sisask},
  plaintextauthor = {Olof Sisask},
  keywords = {VC dimension, uniform continuity, convolutions, regularity lemma},
}   

\dajEDITORdetails{%
   year={2021},
   number={1},
   received={22 April 2020},   
   revised={8 September 2020},    
   published={26 March 2021},  
   doi={10.19086/da.18561},       
}   

\usepackage{amsmath,amsthm,amssymb}
\usepackage{cleveref} 
\usepackage{thmtools}
\declaretheorem[numberwithin=section]{theorem}
\declaretheorem[sibling=theorem]{proposition}
\declaretheorem[sibling=theorem]{lemma}
\declaretheorem[sibling=theorem]{corollary}

\declaretheorem[sibling=theorem,name=Definition,style=definition]{defn}

\declaretheorem[sibling=theorem,style=definition]{remark}

\def\E{\mathbb{E}}
\def\Z{\mathbb{Z}}
\def\R{\mathbb{R}}

\def\C{\mathbb{C}}
\def\N{\mathbb{N}}
\def\P{\mathbb{P}}
\def\F{\mathbb{F}}

\def\Ghat{\widehat{G}}
\def\fhat{\widehat{f}}
\providecommand{\abs}[1]{\lvert#1\rvert}
\providecommand{\Abs}[1]{\left\lvert#1\right\rvert}
\providecommand{\norm}[1]{\lVert #1 \rVert}
\providecommand{\tup}[1]{{\vec{#1}}}
\providecommand{\ceiling}[1]{\lceil#1\rceil}

\DeclareMathOperator{\Bohr}{Bohr}

\DeclareMathOperator{\vcdsym}{dim_{VC}}
\DeclareMathOperator{\vcdrsym}{dim_{rVC}}
\usepackage{mathtools}
\DeclarePairedDelimiterX{\relvcx}[2]{(}{)}{%
  #1\,\delimsize|\,#2%
}
\usepackage{xparse}
\NewDocumentCommand{\vcd}{ >{\SplitArgument{1}{,}} m }{\vcdint#1}
\NewDocumentCommand{\vcdint}{mm}{
\vcdsym\IfNoValueTF{#2}{(#1)}{\relvcx{#1}{#2}}%
}
\NewDocumentCommand{\vcdr}{ >{\SplitArgument{1}{,}} m }{\vcdrint#1}
\NewDocumentCommand{\vcdrint}{mm}{
\vcdrsym\IfNoValueTF{#2}{(#1)}{\relvcx{#1}{#2}}%
}
\DeclareMathOperator{\boxspan}{boxspan}
\DeclareMathOperator{\supp}{supp}
\DeclareMathOperator{\ima}{Im}

\newcommand{\symmdiff}{\mathbin{\triangle}}

\begin{document}

\begin{frontmatter}[classification=text]


\author[olof]{Olof Sisask\thanks{Supported by Swedish Research Council grant 2013-4896}}

\begin{abstract}
We study a notion of VC-dimension for subsets of groups, defining this for a set $A$ to be the VC-dimension of the family $\{ (xA) \cap A : x \in A\cdot A^{-1} \}$. We show that if a finite subset $A$ of an abelian group has bounded VC-dimension, then the convolution $1_A*1_{-A}$ is Bohr uniformly continuous, in a quantitatively strong sense. This generalises and strengthens a version of the stable arithmetic regularity lemma of Terry and Wolf \cite{TeWo:2017} in various ways. In particular, it directly implies that the Polynomial Bogolyubov--Ruzsa Conjecture --- a strong version of the Polynomial Freiman--Ruzsa Conjecture --- holds for sets with bounded VC-dimension. We also prove some results in the non-abelian setting.

In some sense, this gives a structure theorem for translation-closed set systems with bounded (classical) VC-dimension: if a VC-bounded family of subsets of an abelian group is closed under translation, then each member has a simple description in terms of Bohr sets, up to a small error.
\end{abstract}
\end{frontmatter}

\section{Introduction}\label{section:intro}
There are many results in additive combinatorics that attempt to describe the algebraic structure of a set satisfying some combinatorial hypotheses. A popular and general type of hypothesis for a finite subset $A$ of an abelian group is that of small-doubling, taking the form $\abs{A+A} \leq K\abs{A}$, where $A+A \coloneqq \{ a + b : a, b \in A \}$ is the sumset of $A$ with itself. Powerful conclusions are known about such sets: it is known, for example, that in this case the four-fold sumset $4A = A+A+A+A$ must contain a coset of a large subgroup if the ambient group has bounded exponent, say --- such results are associated with the names of Bogolyubov and Ruzsa in the literature. Motivated by model-theoretic considerations, Terry and Wolf \cite{TeWo:2017} recently considered a different type of hypothesis, termed $k$-stability. We will define this in Section \ref{section:AR}, but for now we recall their main theorem, termed an \emph{arithmetic regularity lemma} due to its relation to a more general statement of this name due to Green \cite{Gr:2005}.

\begin{theorem}[Terry--Wolf stable arithmetic regularity]\label{thm:TW}
Let $\epsilon \in (0,1)$, let $k \geq 2$ and let $p$ be a prime. Suppose $G = \F_p^n$ with $n \geq n_0(\epsilon, k, p)$, and that $A \subset G$ is $k$-stable. Then there is a subspace $H \leq G$ of codimension at most $O_k\left(\epsilon^{-O_k(1)}\right)$ such that for each $x \in G$, either $\abs{A \cap (x+H)} \leq \epsilon \abs{H}$ or $\abs{A \cap (x+H)} \geq (1-\epsilon) \abs{H}$. In other words, each coset of $H$ is either almost disjoint from $A$ or almost contained in $A$. Moreover, there is a union $W$ of cosets of $H$ such that $\abs{ A \symmdiff W } \leq \epsilon \abs{G}$. 
\end{theorem}

The point here is that the quantitative aspects of the conclusion are far, far stronger than what one obtains for similar statements without the $k$-stability condition, these being necessarily of tower-type with the height of the tower increasing with $1/\epsilon$, as proved by Green \cite{Gr:2005}, akin to the bound proved by Gowers \cite{Gowers:tower} in the graph regularity context; we refer the reader to \cite{TeWo:2017} for more background. We shall show that a similarly strong quantitative regularity conclusion holds if $A$ satisfies a weaker hypothesis than stability, and, moreover, we shall prove some more general results on the uniform continuity of certain convolutions. To state these results, we introduce our key definition: 
\begin{defn}
Let $G$ be a group and let $A, B \subset G$. We define the \emph{VC-dimension of $A$ relative to $B$} to be the VC-dimension of the family $\{ (x A) \cap B : x \in B \cdot A^{-1} \}$, and denote this by $\vcd{A,B}$. If $B = A$ we write just $\vcd{A}$ and call this the \emph{VC-dimension of $A$}.
\end{defn}

Here, the VC-dimension $\vcd{\mathcal{F}}$ of a family $\mathcal{F}$ of sets is a classical concept in statistics and theoretical computer science, named for Vapnik and Chervonenkis \cite{VC} and defined as the largest size of a set $X$ such that $\mathcal{F}$ \emph{shatters} $X$, ie for which every subset $Y \subset X$ can be written in the form $Y = X \cap F$ for some $F \in \mathcal{F}$. Also, $B \cdot A^{-1} \coloneqq \{ ba^{-1} : a \in A,\ b \in B \}$ is the non-abelian version of set addition/subtraction. We remark that $\vcd{A}$ is almost the same thing as the VC-dimension of the family $\{ tA : t \in G\}$ of all translates of $A$ (see Proposition \ref{prop:monotonicity}), but has somewhat nicer additive combinatorial properties.

We shall look at this concept of VC-dimension in more depth in Section \ref{section:VCprops}, where we prove that it for example is invariant under Freiman isomorphism, and that the commonly considered `nice' sets in additive combinatorics, such as subgroups, generalised arithmetic progressions and (more generally) Bohr sets, have small VC-dimension --- as do some other types of sets not usually considered in conjunction with these, such as Sidon sets.

Our first results provide partial converses to these assertions about examples. We start with a statement about the finite field case, as considered in Theorem \ref{thm:TW}:\footnote{See the end of this section for a word on notation, in particular on the use of the letter $C$.}

\begin{theorem}[{name=VC-bounded arithmetic regularity,restate=[name=]AR}]\label{thm:AR}
Let $\epsilon \in (0,1)$, and let $G = \F_q^n$. If $A \subset G$ has size least $\alpha \abs{G}$ and $\vcd{A} \leq d$, then there is subspace $H \subset A-A$ of codimension at most $C d \epsilon^{-C}\log(2/\alpha)$ and a union $W$ of cosets of $H$, contained in $A+H$, such that $\abs{ A \symmdiff W } \leq \epsilon \abs{A}$.
\end{theorem}

As we show in Section \ref{section:stability}, having bounded VC-dimension is a strictly weaker hypothesis than that of stability: if $A$ is $k$-stable in the sense of Terry--Wolf then $\vcd{A} \leq k-1$, but there are sets of bounded VC-dimension that are not boundedly stable. The above theorem thus applies more generally than Theorem \ref{thm:TW}, but is weaker in not offering control on all the individual translates of $H$. Indeed, such control can fail in the bounded VC-dimension setup, but something like it can be obtained for a large collection of the possible $x+H$ --- see Theorem \ref{thm:ARBohr} and Remarks \ref{rmk:translateControl} and \ref{rmk:translateControlCPT}. We note that a density assumption like ours is implicit in Theorem \ref{thm:TW}, as the conclusion is trivial if $\epsilon \geq \abs{A}/\abs{G}$, taking $H = G$. Let us also remark that the unspecified exponent of $\epsilon^{-1}$ above is very reasonable: certainly one can take $C = 4+o(1)$ (but see Section \ref{section:aftermath} for improvements).

We shall in fact prove a more general result, valid for arbitrary finite abelian groups. For the definitions surrounding Bohr sets, see Section \ref{section:proofAbelian}; for a technically more complete statement, see Section \ref{section:AR}, and for improved bounds, see Section \ref{section:aftermath}.

\begin{theorem}[VC-bounded arithmetic regularity, simplified Bohr set version]\label{thm:ARBohrSimple}\hfill\ \\
Let $\epsilon \in (0,1)$, and let $G$ be a finite abelian group. If $A \subset G$ has size least $\alpha \abs{G}$ and $\vcd{A} \leq d$, then there is a Bohr set $H \subset A-A$ of rank $m \leq C d \epsilon^{-C}\log(2/\alpha)$ and radius at least $c\epsilon^2/m^2$, and a subset $A' \subset A$, such that $\abs{ A \symmdiff (A'+H) } \leq \epsilon \abs{A}$.\\
\end{theorem}

One way of interpreting this result is as a structure theorem for translation-closed set families with bounded VC-dimension: if a set system $\mathcal{F}$ consists of subsets of an abelian group and has bounded VC-dimension, and the translates of each set in the system also lie in the system, then each of the members of $\mathcal{F}$ can be efficiently covered by translates of Bohr sets. Thus, the prototypical examples of set systems with small VC-dimension --- intervals in $\R$ and boxes in $\R^d$ (see Section \ref{section:VCprops}) --- are in some sense canonical.

A group-theoretic definition of VC-dimension very similar to ours was used in a paper of P. Simon \cite{Sim:2017} and was also arrived at from a graph-theoretic perspective by Alon, Fox and Zhao \cite{AFZ:2018}, who independently from this work proved a result along the lines of Theorem \ref{thm:AR}. The approach in \cite{AFZ:2018} has the advantage of leading to a better dependence on $\epsilon$ in the bounded exponent setting, with a logarithmic dependence instead of polynomial, whereas the approach of this paper works for arbitrary finite abelian groups (and is different even in the bounded exponent setting). A lemma from \cite{AFZ:2018} can however be used to improve the dependence on $\epsilon$ in Theorem \ref{thm:ARBohrSimple} to logarithmic; we elaborate on this in Section \ref{section:aftermath}. A further different approach to this theory was given independently around the same time as this work by Conant--Pillay--Terry \cite{CoPiTe:2018}, who used model-theoretic tools to prove a very general result along the lines of Theorems \ref{thm:ARBohrSimple} and \ref{thm:ARBohr}, valid for all (not necessarily abelian) groups, albeit with ineffective bounds; see \cite[Theorem 5.7]{CoPiTe:2018}. An earlier paper \cite{CoPiTe:2017} of the same authors proves an extension of Theorem \ref{thm:TW} to arbitrary groups; we make some further remarks on this in Section \ref{section:stability}.

We shall deduce the above arithmetic regularity lemmas from results on the continuity of convolutions, which are the natural results to which our methods lead. To state these results properly, let us introduce some notation.

For a group $G$ and functions $f, g : G \to \C$, we define their (left) convolution to be
\[ f*g(x) = \sum_{y \in G} f(y)g(y^{-1}x), \]
provided this is well-defined. Denoting the indicator function of a set $X$ by $1_X$, a close connection between our definition of VC-dimension and convolution comes from the fact that $1_B*1_{A^{-1}}(x) = \abs{ (xA) \cap B }$. Note in particular that $B \cdot A^{-1} = \supp(1_B * 1_{A^{-1}})$. It will be convenient to normalise certain sums, and for this purpose we write $\mu_A = 1_A/\abs{A}$ for finite sets $A$. We also extend this to define measures: $\mu_A(X) \coloneqq \abs{A \cap X}/\abs{A}$. We furthermore define the translation operator $\tau_t$ for $t \in G$ by $\tau_t f(x) = f(tx)$.

It is well-known that convolutions of indicator functions (and more general functions) are somewhat smooth, particularly if the sets satisfy some combinatorial condition like small doubling; see for example \cite{Bo:1990, CrSi:2010} or \cite[Theorems 4.1, 6.1]{CrLaSi:2013}, where a notion of $L^p$-smoothness is proved for $p \geq 2$, or \cite[Theorems 5.1, 5.4]{ScSi:2016}, where $L^\infty$-smoothness is proved for convolutions of three sets. Our main results say that if a set has bounded VC-dimension, then one in fact has $L^\infty$-smoothness, ie uniform continuity --- or uniform almost-periodicity --- even for a convolution of two sets. Our first such result is valid for arbitrary groups:
{
\begin{theorem}\label{thm:mainSum}
Let $\epsilon \in (0,1]$ and $d, k \in \N$. Let $G$ be a group and let $A, B \subset G$ be finite subsets with $\vcd{A,B} \leq d$. If $\abs{S \cdot B} \leq K\abs{B}$ for some set $S \subset G$, then there is a set $T \subset S$ of size at least $0.99 K^{-C dk^2/\epsilon^2} \abs{S}$ such that, for each $t \in (T^{-1} T)^k$,
\[ \norm{ \tau_t(\mu_B* 1_{A^{-1}}) - \mu_B* 1_{A^{-1}}}_\infty \leq \epsilon. \]
\end{theorem}
}
Here $X^k$ denotes $X \cdot X \cdots X$ with $k$ copies of $X$. 

We note that if $G$ is finite and $B$ has size at least $\alpha \abs{G}$, then one can take $S = G$ and $K = 1/\alpha$. Similarly, if $B \subset [N] \coloneqq \{1,2,\ldots,N \} \subset \Z$ has $\abs{B} \geq \alpha N$, then one can take $S = [N]$ and $K = 2/\alpha$. This small sumset condition is thus a generalisation of a density condition.

In other words, if $A$ has bounded VC-dimension relative to a set of small doubling, then, averaged or `projected' over this set, $A$ is almost translation-invariant.

For abelian groups, we can impose some structure on the set of almost-periods, ensuring uniform continuity in the so-called Bohr topology (in a quantitatively strong sense); we again refer to Section \ref{section:proofAbelian} for the definitions of the terms involved.

\begin{theorem}\label{thm:mainAbelian}
Let $\epsilon \in (0,1]$ and $d \in \N$. Let $G$ be a finite abelian group and let $A, B \subset G$ be subsets with $\vcd{A,B} \leq d$ and $\eta \coloneqq \abs{B}/\abs{A}$. 
 If $\abs{B + S} \leq K\abs{B}$ for some set $S \subset G$, then there is a regular Bohr set $T$ of rank $m \leq C d \epsilon^{-2} (\log K) (\log 2/\epsilon\eta)^2 + C\log(2\abs{G}/\abs{S})$ and radius at least $c\epsilon \eta^{1/2}/m$ such that, for each $t \in T$,
\[ \norm{\mu_B* 1_{-A}(\cdot + t) - \mu_B* 1_{-A}}_\infty \leq \epsilon. \]
\end{theorem}

We again refer to Section \ref{section:aftermath} for an improved $\epsilon$-dependence in variants of the above theorems, in the light of a lemma from \cite{AFZ:2018}.

Taking $B = A$ and $\epsilon = \tfrac{1}{2}$, and using the fact that $\mu_A* 1_{-A}(0) = 1$ and $\supp(\mu_A*1_{-A}) = A-A$, Theorem \ref{thm:mainAbelian} has as an essentially immediate consequence the following Bogolyubov--Ruzsa-type corollary, which we restrict to the finite field setup for clarity of exposition. We employ here some non-standard terminology: 
\begin{defn}
We say that a set $H \subset \F_q^n$ has \emph{$A$-codimension at most $d$} if $\abs{H} \geq q^{-d} \abs{A}$.
\end{defn}

\begin{corollary}[{restate=[name=]PBR}]\label{cor:PBR}
Let $A \subset \F_q^n$ be a set with $\vcd{A} \leq d$ and $\abs{A-A} \leq K \abs{A}$. Then $A-A$ contains a subspace of $(A-A)$-codimension at most $C d\, \log K + C \log q$.
\end{corollary}

Thus the so-called Polynomial Bogolyubov--Ruzsa Conjecture holds for sets of bounded VC-dimension in the finite field setting, even in the strong sense of working with only the difference set $A-A$ instead of something like $2A-2A$. (See \cite{Sa:bogolyubov, Sa:struct} for the best known bounds towards the general conjecture, due to Sanders.) We note that this conjecture is in general stronger than the well-known Polynomial Freiman--Ruzsa Conjecture \cite{Gr:PFR}.

In the dense regime where $\abs{A} \geq \alpha \abs{G}$, we have the following version.

\begin{corollary}\label{cor:densePBR}
Let $A \subset \F_q^n$ be a set with $\vcd{A} \leq d$ and $\abs{A} \geq \alpha \abs{\F_q^n}$. Then $A-A$ contains a subspace of codimension at most $C d \log(1/\alpha)$.
\end{corollary}

For general sets $A$ of density at least $\alpha$, the best known results \cite{Sa:Greens, CrLaSi:2013} on subspaces in $A-A$ say roughly that $A-A$ contains a subspace of \emph{dimension} at least $\alpha n$, whereas in the bounded VC-dimension-case we instead get a strong upper bound on \emph{codimension}. The so-called niveau set example of Ruzsa \cite{Ru:1990}, extended to the finite field setting by Green \cite[Theorem 9.4]{Gr:FF} (see also \cite[Theorem 1.3]{Sa:Greens}), implies that one could not hope for more than codimension $c\sqrt{n}$ in general, even for $\alpha$ close to $1/2$.


\subsection*{Paper layout}
In Section \ref{section:mainproof} we prove our main general theorem, Theorem \ref{thm:mainSum}, followed in Section \ref{section:proofAbelian} by a review of Bohr sets and a proof of the abelian version, Theorem \ref{thm:mainAbelian}. In Section \ref{section:VCprops} we establish some basic properties of VC-dimension and analyse examples. Section \ref{section:AR} contains our applications: the proofs of the arithmetic regularity lemmas and Corollaries \ref{cor:PBR} and \ref{cor:densePBR}. In Section \ref{section:stability} we note some relationships between $k$-stability and VC-dimension. In Section \ref{section:aftermath} we show how one can obtain an improved dependence on the parameter $\epsilon$ in our results. Finally, we end with some remarks in Section \ref{section:conclusion}.

\subsection*{Some notation}
Throughout the paper, we employ the very convenient `constantly changing constant' device, meaning that, unless otherwise specified, the letters $c$ and $C$ denote positive absolute constants that can vary from occurrence to occurrence, and can be picked to make the statements true and the arguments work. In the statements of results, we always assume that the input-sets are non-empty. In the context of non-abelian groups, we write $X^{\otimes k}$ for the Cartesian product $X \times X \times \cdots \times X$, to distinguish it from the iterated product set $X^k$. Whenever we speak of the finite field $\F_q$, we allow $q$ to be an arbitrary prime power. Finally, for two sets $A$ and $B$, we denote their symmetric difference by $A \symmdiff B = (A \setminus B) \cup (B \setminus A)$.

\subsection*{Subspaces vs subgroups}
Some results in this paper are about vector spaces $\F_q^n$, where $\F_q$ is allowed to be an arbitrary finite field. If $q = p^m$ for a prime $p$, then $\F_q^n \cong \F_p^{mn}$ as groups, and if one is interested only in the group structure, for example in finding a bounded index subgroup, then the quantitative conclusions from some of our results are stronger if one applies them to the group $\F_p^{mn}$ rather than $\F_q^n$. If one however is interested in subspaces rather than subgroups, which for $\F_q^n$ can be different objects, then it is more natural to work directly with $\F_q^n$ instead.

\section{Proof of $L^\infty$-almost-periodicity for arbitrary groups}\label{section:mainproof}

Here we prove Theorem \ref{thm:mainSum}. Our argument largely follows the probabilistic method employed in \cite{CrSi:2010}, incorporating a Glivenko--Cantelli-type uniform law of large numbers valid in the bounded VC-dimension setup:

\begin{theorem}\label{thm:uniformEmp}
Let $\mathcal{X}$ be a set endowed with a probability measure $\mu$, and let $x_1, \ldots, x_n$ be independent $\mathcal{X}$-valued random variables distributed according to $\mu$. Let $\mathcal{A}$ be a countable class of measurable subsets of $\mathcal{X}$, and suppose $\vcd{\mathcal{A}} \leq d$, where $d \geq 1$. Then, provided $n \geq Cd/\epsilon^2$,
\[ \P\left( \sup_{A \in \mathcal{A}} \Abs{\E_{j\in [n]} 1_A(x_j) - \mu(A)} \leq \epsilon \right) \geq 0.99. \]
\end{theorem}

In other words, for collections of sets with bounded VC-dimension, the empirical mean is with high probability a uniformly good estimator for the measure. This type of result is well known in the theory of empirical processes, and this particular version can be read out of the paper \cite{Talagrand94} of Talagrand. Indeed, in \cite[Theorem 1.1]{Talagrand94}, take $\mathcal{F} = \{ 1_A : A \in \mathcal{A} \}$ and $M = \epsilon\sqrt{n}$, with $V = 4e$ as mentioned in the discussion after the theorem. Writing $n = Ld/\epsilon^2$ where $L \geq C$, the conclusion of the theorem implies that 
\[ \P\left( \sup_{A \in \mathcal{A}} \Abs{\E_{j\in [n]} 1_A(x_j) - \mu(A)} \geq \epsilon \right) \leq e^{-d(2L - \log(CL))}, \]
which for $L \geq C$ is at most $0.01$, giving the result. See also \cite[\S 13.3]{BoLuMa} for an alternative presentation. In either case, the argument makes use Haussler's universal entropy bound \cite{Haussler}. See Section \ref{section:conclusion} for a brief discussion of looser conditions than bounded VC-dimension under which the same conclusion holds.

Using this, we are ready to prove Theorem \ref{thm:mainSum}, which follows immediately from the following version with $k = 1$, by the triangle inequality and translation invariance. 

\begin{theorem}
Let $\epsilon \in (0,1]$ and $d \in \N$. Let $G$ be a group and let $A, B \subset G$ be finite subsets with $\vcd{A,B} \leq d$. If $\abs{S \cdot B} \leq K\abs{B}$ for some set $S \subset G$, then there is a set $T \subset S$ of size at least $0.99 K^{-C d/\epsilon^2} \abs{S}$ such that, for each $t \in T^{-1} T$,
\[ \norm{ \tau_t(\mu_B* 1_{A^{-1}}) - \mu_B* 1_{A^{-1}}}_\infty \leq \epsilon. \]
\end{theorem}
\begin{proof}
Let $n = Cd/\epsilon^2$ and let $\tup{s} \in B^{\otimes n}$ be sampled uniformly at random. We write
\[ \mu_\tup{s}* 1_{A^{-1}}(x) = \E_{j \in [n]} 1_A(x^{-1} s_j) = \E_{j \in [n]} 1_{(xA) \cap B}(s_j), \]
a random variable indexed by $x$, with expectation $\mu_B* 1_{A^{-1}}(x)$. By Theorem \ref{thm:uniformEmp}, with the class $\mathcal{A} = \{ (xA) \cap B : x \in B \cdot A^{-1} \}$, the probability measure $\mu_B(X) = \abs{B \cap X}/\abs{B}$ on $\mathcal{X} = G$, and the random variables $x_i = s_i$ we have that 
\[ \sup_{x \in B\cdot A^{-1}}\Abs{ \E_{j \in [n]} 1_{(xA) \cap B}(s_j) - \mu_B(xA) } \leq \tfrac{1}{2}\epsilon \]
holds with probability at least $0.99$. Since all terms are $0$ unless $x$ lies in $B \cdot A^{-1}$, we may extend the supremum to run over the whole of $G$. Noting that $\mu_B(xA) = \mu_B * 1_{A^{-1}}(x)$ we thus have that, with probability at least $0.99$,
\begin{equation}
\norm{ \mu_\tup{s}* 1_{A^{-1}} - \mu_B* 1_{A^{-1}} }_{\infty} \leq \tfrac{1}{2} \epsilon. \label{eqn:good}
\end{equation}

We may now follow the averaging argument of \cite{CrSi:2010} (see the proof of Theorem 3.1 there) to obtain the set $T$; we include the details (of a slight variant) for completeness. Call the tuples $\tup{s} \in G^{\otimes n}$ satisfying \eqref{eqn:good} \emph{good}, so that $\P_{\tup{s} \in B^{\otimes n}}(\text{$\tup{s}$ is good}) \geq 0.99$. Define
\[ T_\tup{s} = \{ t \in S : \text{ $t^{-1}\tup{s}$ is good} \} \subset S. \]
Then
\begin{align*}
\E_{\tup{s} \in (S \cdot B)^{\otimes n}} \abs{T_\tup{s}} &= \sum_{t \in S} \P_{\tup{s} \in (S \cdot B)^{\otimes n}}\left(\text{$t^{-1}\tup{s}$ is good}\right) \\
&\geq \frac{\abs{B}^n}{\abs{S\cdot B}^n} \P_{\tup{s} \in B^{\otimes n}}\left(\text{$\tup{s}$ is good}\right)\abs{S} \\
&\geq 0.99 K^{-n}\abs{S}.
\end{align*}
We now fix some $\tup{s}$ such that $T_\tup{s}$ has at least this size, and let this set be our $T$.

We next show that every $t \in T^{-1}T$ is an almost-period, completing the proof. Indeed, if $t = r^{-1}s$ with $r, s \in T$, then, since $r^{-1}\tup{s}$ is good,
\[ \norm{ \tau_r(\mu_{\tup{s}}* 1_{A^{-1}}) - \mu_B* 1_{A^{-1}} }_\infty \leq \tfrac{1}{2} \epsilon, \]
and similarly for $s$. Thus
\begin{align*}
\norm{ \tau_t(\mu_B* 1_{A^{-1}}) - \mu_B* 1_{A^{-1}} }_\infty &= \norm{ \tau_s \tau_{r^{-1}}(\mu_B* 1_{A^{-1}}) - \mu_B* 1_{A^{-1}} }_\infty \\
&\leq \norm{ \tau_s \tau_{r^{-1}}(\mu_B* 1_{A^{-1}}) - \tau_s(\mu_\tup{s}* 1_{A^{-1}}) }_\infty \\
&\qquad + \norm{ \tau_s (\mu_\tup{s}* 1_{A^{-1}}) - \mu_B* 1_{A^{-1}} }_\infty \\
&= \norm{ \mu_B* 1_{A^{-1}} - \tau_r(\mu_\tup{s}* 1_{A^{-1}}) }_\infty + \norm{ \tau_s (\mu_\tup{s}* 1_{A^{-1}}) - \mu_B* 1_{A^{-1}} }_\infty \\
&\leq \epsilon. \qedhere
\end{align*}
\end{proof}

\begin{remark}
We could also have used the theory surrounding randomly sampled $\epsilon$-approximations here, which involves sampling without replacement instead, as done in \cite{CrSi:2010}.
\end{remark}

As noted earlier, Theorem \ref{thm:mainSum} follows immediately by the triangle inequality.

\section{Bohr sets and the abelian case}\label{section:proofAbelian}
We now bootstrap Theorem \ref{thm:mainSum} to show Theorem \ref{thm:mainAbelian}, employing a simple Fourier-analytic argument. Let us first introduce some terminology and notation.

\begin{defn}
For an abelian group $G$, write $\Ghat = \{ \gamma : G \to S^1 : \text{ $\gamma$ a homomorphism} \}$ for the \emph{dual group} of $G$, consisting of homomorphisms from $G$ to the unit circle in $\C$, endowed with pointwise multiplication of functions as the group law. The elements of $\Ghat$ are called \emph{characters}. For a subset $\Gamma \subset \Ghat$ and a real $\rho \geq 0$, we define the \emph{Bohr set} on these data by
\[ \Bohr(\Gamma, \rho) = \{ x \in G : \abs{\gamma(x) - 1} \leq \rho \text{ for all $\gamma \in \Gamma$} \}, \]
and call $\abs{\Gamma}$ the \emph{rank} and $\rho$ the \emph{radius} of the Bohr set. For $B = \Bohr(\Gamma, \rho)$ and $\tau \geq 0$, we write $B_\tau = \Bohr(\Gamma, \tau\rho)$ for the Bohr set with its radius dilated by $\tau$.
\end{defn}

For the basics surrounding Bohr sets, the reader may consult \cite[Chapter 4]{TV}.\footnote{Note however that in \cite{TV}, the dual group $\Ghat$ is identified with $G$, these being isomorphic for finite abelian groups, and that Bohr sets are defined there in terms of $\arg(\gamma(x))$ being close to $0$ instead of $\gamma(x)$ being close to $1$.}  We note the following standard results, meant to illustrate that Bohr sets are large and structured.

\begin{lemma}
Let $G$ be a finite abelian group, and let $B \subset G$ be a Bohr set of rank $d$ and radius $\rho \leq 1$. Then $\abs{B} \geq \left(\tfrac{1}{2\pi}\rho \right)^d \abs{G}$.
\end{lemma}
This follows from \cite[Lemma 4.20]{TV}, with the rescaling in $\rho$ coming from the fact that we are measuring the distance from $\gamma(x)$ to $1$ instead of the distance from $\arg(\gamma(x))$ to $0$.

We include the proof of the following result as an illustrative example of Bohr sets being structured.
\begin{lemma}\label{lemma:BohrStructure}
Let $G$ be an abelian group, and let $B \subset G$ be a Bohr set of rank $d$. If $G$ has exponent $r$, then $B$ contains a subgroup of $G$ of index at most $r^d$. If $G = \F_q^n$ is a vector space over $\F_q$, then $B$ contains a subspace of codimension at most $d$.
\end{lemma}
Recall the distinction between subgroups and subspaces made at the end of Section \ref{section:intro}.

\begin{proof}
For the first claim, note that for every $\gamma \in \Ghat$ we have $\gamma(x)^r = \gamma(rx) = \gamma(0) = 1$, whence every $\gamma$ takes values in the $r$-th roots of unity $U_r \leq \C^\times$. The map $\varphi : G \to U_r^\Gamma$ defined by $x \mapsto (\gamma(x))_{\gamma \in \Gamma}$ is a homomorphism whose kernel is $\ker(\varphi) = \Bohr(\Gamma, 0) \subset B$, and $G/\ker(\varphi) \cong \ima(\varphi) \leq U_r^\Gamma$.

For the second claim, recall that each character $\gamma$ on $\F_q^n$ has the form $\gamma(x_1,\ldots,x_n) = e^{2\pi i \text{Tr}(\xi_1 x_1 + \cdots + \xi_n x_n)/p}$ for some $\xi_j \in \F_q$, where $p$ is the characteristic of $\F_q$. The kernel $\ker(\gamma)$ of such a character thus contains a subspace of codimension at most $1$, given by the condition $\xi_1 x_1 + \cdots + \xi_n x_n = 0$. Since $B$ contains $\bigcap_{\gamma \in \Gamma} \ker(\gamma)$, it contains a subspace of codimension at most $d$.
\end{proof}

For groups without a nice subgroup structure, for example the groups $\Z/N\Z$ of prime order, Bohr sets still contain very naturally structured sets, such as long arithmetic progressions and dense generalised arithmetic progressions (see Definition \ref{defn:GAP}). The general case is a mix of these two: Bohr sets contain large so-called \emph{coset progressions}. We will not elaborate on this here; see \cite[Chapter 4]{TV} for further details. We do, however, require one further technical aspect of Bohr sets.

\begin{defn}[Regular Bohr sets]
We say that a Bohr set $B$ of rank $d$ is \emph{regular} if 
\[ 1 - 12 d \abs{\tau} \leq \frac{\abs{B_{1+\tau}}}{\abs{B}} \leq 1 + 12 d \abs{\tau} \]
whenever $\abs{\tau} \leq 1/12d$.
\end{defn}

We shall apply this through the following direct consequence:

\begin{lemma}\label{lemma:regConv}
If $B$ is a regular Bohr set of rank $d$ and $t \in B_\delta$ with $\delta \leq \epsilon/24d$, then
\[ \sum_{x \in G} \abs{\mu_B(x + t) - \mu_B(x)} \leq \epsilon. \]
\end{lemma}

Not all Bohr sets are regular in this way, but regular dilates are plentiful:
\begin{lemma}\label{lemma:bohrReg}
If $B$ is a Bohr set, then there is a $\tau \in [\frac{1}{2}, 1]$ for which $B_\tau$ is regular.
\end{lemma}
One can prove this in the same way as \cite[Lemma 4.25]{TV}, just taking into account the aforementioned difference in conventions for defining Bohr sets.

For our proof of Theorem \ref{thm:mainAbelian}, we shall use the Fourier transform:
\begin{defn}
Let $G$ be a finite abelian group. We define the Fourier transform of a function $f : G \to \C$ to be the function $\fhat : \Ghat \to \C$ given by
\[ \fhat(\gamma) = \sum_{x \in G} f(x)\overline{\gamma(x)}. \]
\end{defn}

With this normalisation, the Fourier inversion formula, convolution identity and Parseval's identity take the form
\[ f = \E_{\gamma} \fhat(\gamma) \gamma, \qquad \widehat{f*g} = \fhat \cdot \widehat{g}, \qquad \E_{\gamma} \abs{\fhat(\gamma)}^2 = \norm{f}_{2}^2, \]
where $\E_{\gamma} = \tfrac{1}{\abs{\Ghat}} \sum_{\gamma \in \Ghat}$ and $\norm{f}_2^2 = \sum_{x \in G} \abs{f(x)}^2$.

We now perform the bootstrapping mentioned at the start of this section. The idea is to convolve $\mu_B* 1_{-A}$ with an iterated convolution of the set of almost-periods from Theorem \ref{thm:mainSum}, using Chang's lemma to control the dimension. Similar arguments are used in \cite{Sa:bogolyubov, CrLaSi:2013, ScSi:2016}; see for example the proof of \cite[Theorem 5.4]{ScSi:2016}.

\begin{proof}[Proof of Theorem \ref{thm:mainAbelian}]
The result is trivial if $\eta > 2\epsilon^{-1}$, with $T = G$, since $\norm{{\mu_B* 1_{-A}}}_\infty \leq 1/\eta$, so we assume $\eta \leq 2\epsilon^{-1}$. We apply Theorem \ref{thm:mainSum} with parameters $\epsilon/3$ and $k = \ceiling{ C\log(2/\epsilon\eta^{1/2}) }$, giving us a set $X \subset S$ of size at least $0.99 K^{-C dk^2/\epsilon^2}\abs{S}$ such that, for each $x \in G$ and $t \in kX-kX$,
\begin{equation}
\abs{ \mu_B* 1_{-A}(x+t) - \mu_B* 1_{-A}(x) } \leq \tfrac{1}{3}\epsilon.\label{eqn:pre}
\end{equation}
Let $\mu = \mu_X^{(k)} * \mu_{-X}^{(k)}$, where $\mu_R^{(k)} = \mu_R*\mu_R *\cdots *\mu_R$ denotes the $k$-fold convolution of $\mu_R$ with itself, so that $\mu$ is supported on $kX-kX$, is symmetric and satisfies $\sum_{t \in G} \mu(t) = 1$. 
The triangle inequality coupled with \eqref{eqn:pre} then gives, for any $x \in G$,
\[ \abs{ \mu_B* 1_{-A}* \mu(x) - \mu_B* 1_{-A}(x) } = \abs{ \sum_{t \in kX-kX} \mu(t)( \mu_B*1_{-A}(x+t) - \mu_B*1_{-A}(x) ) } \leq \tfrac{1}{3}\epsilon. \]
Thus, for any $x,t \in G$ we have by a couple more applications of the triangle inequality that
\begin{align}
\abs{ \mu_B* 1_{-A}(x+t) - \mu_B* 1_{-A}(x) } &\leq \abs{ \mu_B* 1_{-A}(x+t) - \mu_B* 1_{-A} * \mu(x+t) } \nonumber \\
&\qquad + \abs{ \mu_B* 1_{-A}*\mu(x+t) - \mu_B* 1_{-A} * \mu(x) } \nonumber \\
&\qquad\qquad + \abs{ \mu_B* 1_{-A} * \mu(x) - \mu_B* 1_{-A}(x) } \nonumber \\
&\leq \tfrac{2}{3}\epsilon + \abs{ \mu_B* 1_{-A} * \mu(x+t) - \mu_B* 1_{-A}* \mu(x) }.\label{eqn:fourierTriangle}
\end{align}
Our task thus reduces to finding a set $T$ for which this last term is small for all $t \in T$.

Let us write $\Gamma = \{ \gamma \in \Ghat : \abs{\widehat{\mu_X}(\gamma)} \geq \tfrac{1}{2} \}$ for the large spectrum of $\mu_X$. Then for any $t \in \Bohr\left(\Gamma, \tfrac{1}{3}\epsilon \eta^{1/2}\right)$ and any $x \in G$ we have
\begin{align}
\abs{ \mu_B* 1_{-A}*\mu(x+t) - \mu_B* 1_{-A}*\mu(x) } &\leq \E_{\gamma \in \Ghat} \abs{\widehat{\mu_B}(\gamma)} \abs{\widehat{1_A}(\gamma)} \abs{\widehat{\mu_X}(\gamma)}^{2k} \abs{\gamma(t)-1},\label{eqn:Fbound}
\end{align}
by the Fourier inversion formula, the convolution identity and the triangle inequality. Splitting this average up according to whether $\gamma \in \Gamma$ or not, bounding 
\[ \abs{\widehat{\mu_X}(\gamma)} \leq \norm{\mu_X}_1 = 1 \text{ and } \abs{\gamma(t)-1} \leq \tfrac{1}{3}\epsilon \eta^{1/2}\quad \text{for $\gamma \in \Gamma$,} \]
and 
\[ \abs{\widehat{\mu_X}(\gamma)} \leq 1/2 \text{ and } \abs{\gamma(t)-1} \leq 2 \quad\text{for $\gamma \notin \Gamma$,} \]
we see that \eqref{eqn:Fbound} is at most
\begin{align*}
\tfrac{1}{3} \epsilon \eta^{1/2} \E_{\gamma \in \Ghat} \abs{\widehat{\mu_B}(\gamma)} \abs{\widehat{1_A}(\gamma)}1_\Gamma(\gamma) + 2^{-2k+1} \E_{\gamma \in \Ghat} \abs{\widehat{\mu_B}(\gamma)} \abs{\widehat{1_A}(\gamma)}1_{\Ghat\setminus \Gamma}(\gamma). 
\end{align*}
Our choice of $k$ ensures that $2^{-2k+1} \leq \tfrac{1}{3} \epsilon \eta^{1/2}$, and so \eqref{eqn:Fbound} is at most
\begin{align*}
\tfrac{1}{3} \epsilon \eta^{1/2} \E_{\gamma \in \Ghat} \abs{\widehat{\mu_B}(\gamma)} \abs{\widehat{1_A}(\gamma)} \leq \tfrac{1}{3} \epsilon \eta^{1/2} \norm{\mu_B}_2 \norm{1_A}_2 = \tfrac{1}{3}\epsilon,
\end{align*}
where the inequality uses the Cauchy-Schwarz inequality and Parseval's identity. Combining this final bound with \eqref{eqn:fourierTriangle}, which was valid for any $t$, we are almost done: it only remains to bound the rank of the Bohr set. By Chang's lemma \cite[Lemma 4.36]{TV}, the spectrum $\Gamma$ (which in the notation of \cite{TV} is $\text{Spec}_{1/2}(X)$; beware also that \cite{TV} uses normalised Fourier transforms) is contained in the $\{0,\pm 1\}$-span of a set of $m \leq C\log(2\abs{G}/\abs{X})$ characters $\Lambda$. We thus take the Bohr set of the conclusion to be $T = \Bohr\left(\Lambda, \tfrac{c}{m}\epsilon \eta^{1/2}\right) \subset \Bohr\left(\Gamma, \tfrac{1}{3}\epsilon \eta^{1/2}\right)$, the constant $c \in [1/6,1/3]$ being picked for regularity using Lemma \ref{lemma:bohrReg}.
\end{proof}

\section{Examples and basic properties of group VC-dimension}\label{section:VCprops}
In this section we note some basic properties of our notion of group-theoretic VC-dimension, and give some examples of types of set with small VC-dimension.

\subsection*{Basic properties of group VC-dimension}

Our notion of $\vcd{A,B}$ is of course not monotone with respect to $A$, but it is with respect to the ground set $B$:

\begin{proposition}[Monotonicity]\label{prop:monotonicity}
Let $A, B$ be subsets of a group $G$.
\begin{enumerate}
\item If $B \subset B'$, then $\vcd{A,B} \leq \vcd{A, B'}$.\label{item:mono}
\item Writing $d = \vcd{\{ (xA) \cap B : x \in G \}}$, we have $d - 1 \leq \vcd{A,B} \leq d$.\label{item:unrestricted}
\item $\vcd{A,G} - 1 \leq \vcd{A} \leq \vcd{A,G}$.
\item If $\vcd{A} \leq d$, then $\vcd{A,B} \leq d+1$ for every $B$.\label{item:vccontrol}
\end{enumerate}
\end{proposition}
\begin{proof}
The first item is immediate, as if $X \subset B$ is shattered by $\{ (xA) \cap B : x \in B\cdot A^{-1} \}$ then it is also shattered by $\{ (xA) \cap B' : x \in B' \cdot A^{-1} \}$. 
The second item follows from the fact that the only potential difference between the corresponding set systems is the empty set. The upper bound in the third item follows from \textit{(\ref{item:mono})}. For the lower bound, suppose $\{ xA : x \in G \}$ shatters a set $X$. Then $X \subset tA$ for some $t \in G$. The family $\{ (xA) \cap A : x \in G \}$ then shatters $t^{-1}X \subset A$, and so the bound follows from the lower bound in \textit{(\ref{item:unrestricted})}. For \textit{(\ref{item:vccontrol})}, $\vcd{A,B} \leq \vcd{A,G} \leq \vcd{A}+1$.
\end{proof}

As one would expect, VC-dimension is translation-invariant:

\begin{lemma}[$\vcdsym$ under translation]
Let $A, B \subset G$. Then $\vcd{At, Bt} = \vcd{A,B} = \vcd{tA, tB}$ for every $t \in G$. In particular, $\vcd{At} = \vcd{A} = \vcd{tA}$.
\end{lemma}
\begin{proof}
We have
\[ \vcd{At, Bt} = \vcd{\{ xAt \cap Bt : x \in B\cdot A^{-1} \}}, \]
and it is apparent that the family $\{ xA \cap B : x \in B\cdot A^{-1} \}$ shatters $X$ iff the above family shatters $Xt$. For left-translation we are similarly done, as
\[ \vcd{tA, tB} = \vcd{\{ tyA \cap tB : y \in B \cdot A^{-1} \}}. \qedhere \]
\end{proof}

In fact, a more general property is true: $\vcdsym$ is invariant under Freiman isomorphism:

\begin{defn}\label{defn:FreimanIso}
Let $A, B$ be subsets of a group $G$, and $C, D$ subsets of a group $H$. A pair of maps $\varphi_A : A \to C$, $\varphi_B : B \to D$ is called a \emph{(Freiman) $2$-isomorphism} if each map is a bijection and it holds for all $a_1,a_2 \in A$ and $b_1,b_2 \in B$ that
\[ a_1 b_1^{-1} = a_2 b_2^{-1} \text{ iff } \varphi_A(a_1)\varphi_B(b_1)^{-1} = \varphi_A(a_2) \varphi_B(b_2)^{-1}. \]
The pairs $(A,B)$ and $(C,D)$ are said to be \emph{(Freiman) $2$-isomorphic} if such a pair of maps exists. If $\varphi_A = \varphi_B = \varphi$, so in particular $A=B$ and $C=D$, then we say that $\varphi$ is a \emph{(Freiman) $2$-isomorphism}, and that $A$ and $C$ are \emph{(Freiman) $2$-isomorphic}.
\end{defn}

Note in particular that a pair of maps as above induces a well-defined bijection $\varphi : B\cdot A^{-1} \to D\cdot C^{-1}$, given by $\varphi(ba^{-1}) = \varphi_B(b) \varphi_A(a)^{-1}$.

Some typical examples of Freiman isomorphisms are translations, dilations in certain contexts, embeddings of subsets of infinite groups into finite ones --- for example embedding a finite subset of $\Z$ into $\Z/N\Z$. A particularly useful application concerns embedding a subset $A$ of an abelian group satisfying $\abs{A+A} \leq K\abs{A}$ into a finite abelian group $G$ where $\abs{A} \geq \abs{G}/C(K)$, that is where the isomorphic copy is dense --- this is known in the literature as \emph{modelling}. We include such a lemma in Section \ref{section:AR}, and otherwise refer the reader to \cite[Chapter 5.3]{TV} for more information on Freiman isomorphisms.

\begin{lemma}[$\vcdsym$ under Freiman isomorphism]\label{lemma:Freiman}
Let $A, B$ be subsets of a group $G$, and $C, D$ subsets of a group $H$. If $(A,B)$ and $(C,D)$ are $2$-isomorphic, then $\vcd{A,B} = \vcd{C,D}$. In particular, if $A$ and $C$ are $2$-isomorphic, then $\vcd{A} = \vcd{C}$.
\end{lemma}
\begin{proof}
Take a pair of maps $\varphi_A, \varphi_B$ giving a $2$-isomorphism. We claim that 
\[ \{  yC \cap D : y \in D \cdot C^{-1} \} = \{ \varphi_B(xA \cap B) : x \in B \cdot A^{-1} \}, \]
and since the latter family is clearly isomorphic to $\{ xA \cap B : x \in B\cdot A^{-1} \}$, this proves the lemma. To prove the claim, let us denote the image of $x \in B \cdot A^{-1}$ under the induced bijection $\varphi$ described above by $y_x$, so that we can rewrite the left-hand side of the claim as $\{ y_x \varphi_A(A) \cap \varphi_B(B) : x \in B \cdot A^{-1} \}$. We will thus be done upon showing that, for each $x = b_x a_x^{-1}$,
\[ y_x \varphi_A(A) \cap \varphi_B(B) = \varphi_B(xA \cap B). \]
To see this, note that $\varphi_B(b) = y_x \varphi_A(a) = \varphi_B(b_x)\varphi_A(a_x)^{-1}\varphi_A(a)$ holds iff $b = b_x a_x^{-1} a = xa$, by the Freiman isomorphism property, whence the claim follows.
\end{proof}

Note that the conclusion $\vcd{A} = \vcd{C}$ in fact holds under the weaker assumption that the pairs $(A,A)$ and $(C,C)$ are $2$-isomorphic, allowing one to pick different maps for the two copies of $A$.

There was an arbitrary choice in our definition of VC-dimension, corresponding to the same choice in our definition of convolution: we mostly consider multiplication acting on the group on the \emph{left}; thus our definition of $\vcd{A,B}$ might more appropriately be termed left VC-dimension. There is similarly a notion of right VC-dimension:

\begin{defn}
For $A,B \subset G$, we write $\vcdr{A,B} = \vcd{\{ Ax \cap B : x \in A^{-1} \cdot B \}}$, and just $\vcdr{A}$ if $B = A$, and call these \emph{right VC-dimensions}.
\end{defn}

Clearly $\vcdsym$ and $\vcdrsym$ coincide for abelian groups. In general we have the following relationship.

\begin{lemma}[$\vcdsym$ under inverses]
Let $A, B \subset G$. Then $\vcd{A^{-1}, B^{-1}} = \vcdr{A, B}$. In particular, if $G$ is abelian then $\vcd{-A} = \vcd{A}$.
\end{lemma}
\begin{proof}
One of the families shatters a set $X$ iff the other shatters $X^{-1}$.
\end{proof}
Note that inversion is in general not a Freiman isomorphism, though it is for abelian groups.

\subsection*{Examples of sets with small VC-dimension}
Sets with VC-dimension $0$ have a straightforward description:

\begin{proposition}[Subgroups]\label{prop:vcd0}
Let $A \subset G$ be non-empty. Then $\vcd{A} = 0$ if and only if $A$ is a (left or right) coset of a subgroup.
\end{proposition}
\begin{proof}
A set system $\mathcal{A}$ has VC-dimension $0$ iff $\abs{\mathcal{A}} = 1$. Thus $\vcd{A} = 0$ iff 
\begin{equation*}
\{ (xA) \cap A : x \in A\cdot A^{-1} \} = \{ A \}, 
\end{equation*}
since $A \cdot A^{-1}$ contains the identity. This holds iff $A \subset xA$ for all $x \in A\cdot A^{-1}$, which by translation and symmetry of $A \cdot A^{-1}$ holds iff $A = xA$ for all $x \in A \cdot A^{-1}$. This is equivalent to $A = \langle A \cdot A^{-1} \rangle A$, which, fixing any $t \in A$, holds iff $A = \langle A \cdot A^{-1} \rangle t$.
\end{proof}

Groups like $\Z$ do not contain non-trivial finite subgroups, but they do contain plenty of interesting alternatives, like arithmetic progressions. In the below, we use interval notation to denote the corresponding real intervals intersected with $\Z$.

\begin{proposition}[Arithmetic progressions]\label{prop:APs}
Let $A \subset \Z$ be an arithmetic progression, with $\abs{A} \geq 3$. Then $\vcd{A} = 2 = \vcd{A,\Z}$.
\end{proposition}
\begin{proof}
For the first claim, we may translate and dilate the arithmetic progression to assume that $A = [0,N]$ with $N \geq 2$. The relevant set system is then
\[ \{ [0,0], [0,1], [0,2], \ldots, [0, N], [1, N], [2, N], \ldots, [N,N] \}. \]
This shatters the set $\{0,1\}$, and so $\vcd{A} \geq 2$. On the other hand, any collection of intervals cannot shatter a set of three elements, since if $a < b < c$, then any interval containing $a$ and $c$ automatically contains $b$. Thus $\vcd{A,\Z} \leq 2$, and so the result follows by monotonicity.
\end{proof}

\begin{proposition}[Boxes]\label{prop:boxes}
Let $A = [0,N_1]\times \cdots \times [0,N_d] \subset \Z^d$. Then $\vcd{A} \leq \vcd{A, \Z^d} \leq 2d$.
\end{proposition}
This can be sharpened, but we give the above bound as it is simple to prove, and we include the (standard) proof to give an idea of the style of argument.
\begin{proof}
We prove the stronger, and well-known, claim that the family $\mathcal{B}$ of axis-aligned boxes in $\R^d$ has VC-dimension at most $2d$. Define the box-span of a set $X \subset \R^d$ to be 
\[ \boxspan(X) = [ \min_{x \in X} x_1, \max_{x \in X} x_1] \times \cdots \times [ \min_{x \in X} x_d, \max_{x \in X} x_d ]; \]
this is the smallest axis-aligned box containing $X$. If $\mathcal{B}$ shatters a non-empty set $X$, let $Y = \{p_1, \ldots, p_{2d} \} \subset X$ be such that $\boxspan(X) = \boxspan(Y)$ --- such points exist since at most $2d$ numbers define the box-span. Then $\mathcal{B}$ cannot distinguish between $X$ and its subset $Y$, since any axis-aligned box containing $Y$ contains $X$. Thus $Y = X$ and $\abs{X} \leq 2d$.
\end{proof}

This immediately implies a bound for the natural analogue in $\Z$, namely generalised arithmetic progressions:

\begin{defn}\label{defn:GAP}
Let $G$ be an abelian group, and let $x_1, \ldots, x_d \in G$. A \emph{generalised arithmetic progression} of rank $d$ in $G$ is a set $A \subset G$ of the form
\[ A = a + \{ \lambda_1 x_1 + \cdots + \lambda_d x_d : \lambda_i \in [0,N_i) \}. \]
If $\abs{A} = N_1 \cdots N_d$, that is if all the sums are distinct, then $P$ is called \emph{proper}.
\end{defn}

\begin{proposition}[Generalised arithmetic progressions]
Let $A \subset \Z$ be a proper generalised arithmetic progression of rank $d$, with $A-A$ also proper\footnote{Viewed as a GAP with the same basis elements $x_1,\ldots,x_d$ in the natural way.}. Then $\vcd{A} \leq 2d$.
\end{proposition}
\begin{proof}
This follows from Lemma \ref{lemma:Freiman} --- the invariance of $\vcdsym$ under Freiman isomorphism --- and Proposition \ref{prop:boxes}, as the set $A$ as in Definition \ref{defn:GAP} is the image of $[0,N_1)\times\cdots \times [0,N_d)$ under the obvious map, which is a Freiman isomorphism by properness, and so $\vcd{A} = \vcd{{[0,N_1)\times\cdots \times [0,N_d)}} \leq 2d$. 
\end{proof}

We refer the reader to \cite[Chapter 3.2]{TV} for more information on generalised arithmetic progressions.

In fact, all of the above upper bounds follow from the following example, which also shows that properness is not needed above.

\begin{proposition}[Bohr sets]
Let $G$ be an abelian group. If $A$ is a Bohr set of rank $d$ and radius $< \sqrt{2}$ in $G$, then $\vcd{A} \leq \vcd{A, G} \leq 2d$.
\end{proposition}
\begin{proof}
This follows for much the same reason as the box example earlier, as a Bohr set is the inverse image of a $d$-dimensional box under a certain homomorphism. Deducing it from that result appears to be somewhat involved due to lack of properness, so we argue directly. Suppose $A = \Bohr(\Gamma, \delta)$ where $\abs{\Gamma} = d$ and $\delta < \sqrt{2}$. Suppose for a contradiction that the family
\[ \{ t+A : t \in G \} \quad\text{shatters}\quad X, \]
a set of size $\abs{X} \geq 2d+1$. Without loss of generality (by translation) we assume that $X \subset A$, so that
\[ \abs{\gamma(x) - 1} \leq \delta \quad\text{for all $\gamma \in \Gamma$ and all $x \in X$.} \]
In particular, since $\delta < \sqrt{2}$ this means that all the points $\gamma(x)$ lie on the same half-circumference of the unit circle, for every $\gamma$ and $x$. Also, for every $x \in X$ there is, by shattering, some $t_x$ such that
\[ (t_x + A) \cap X = X\setminus\{x\}. \]
Thus $\abs{\gamma(y-t_x) - 1} \leq \delta$ for all $y \in X\setminus\{x\}$ and $\gamma \in \Gamma$, but there is some $\gamma_x \in \Gamma$ for which $\abs{\gamma_x(x-t_x) - 1} > \delta$. Since we have $\abs{X} \geq 2d+1$ choices for $x$ but only $d$ choices for $\gamma_x$, there must be distinct $x_1, x_2, x_3 \in X$ for which these $\gamma_{x_j}$s are all the same, say $\gamma$; let us also call the corresponding elements $t_{x_j}$ just $t_j$. We thus have
\begin{align*}
\abs{\gamma(x_1) - \gamma(t_1)} &> \delta \\
\abs{\gamma(x_2) - \gamma(t_1)} &\leq \delta \\
\abs{\gamma(x_3) - \gamma(t_1)} &\leq \delta
\end{align*}
and the same for any permutation of the indices $1,2,3$. Viewed as points on the unit circle, the above three relations determine that $\gamma(x_1)$ does not lie along the shortest arc joining $\gamma(x_2)$ and $\gamma(x_3)$. By symmetry, the same is true for the other permutations, which is a contradiction, as one of the three points must lie between the other two.
\end{proof}

A different type of example, defined not in terms of a given structure but instead in terms of lacking a particular structure, comes from Sidon sets:

\begin{defn}
A subset $A$ of a group $G$ is called a \emph{Sidon set} if the only solutions to $a_1 a_2^{-1} = a_3 a_4^{-1}$ with $a_i \in A$ are the trivial ones with $a_1 = a_2$ or $a_1 = a_3$.
\end{defn}

\begin{proposition}[Sidon sets]\label{prop:Sidon}
If $A$ is a Sidon set, then $\vcd{A} \leq \vcd{A,G} \leq 2$.
\end{proposition}
\begin{proof}
Suppose for a contradiction that the family $\{tA : t \in G\}$ shatters a set $\{a,b,c\}$ of three elements; wlog $a,b,c \in A$ (by translating if necessary). Then there is some element $t$ such that
\[ (tA) \cap \{a,b,c\} = \{a,b\}. \]
Hence there are $x,y \in A$ such that $tx = a$ and $ty = b$. Hence $ax^{-1} = by^{-1}$, and since $A$ is a Sidon set and $a$ and $b$ are distinct we have $a = x$, and so $t = \text{id}$. But this contradicts the fact that $c \notin tA$, and we are done.
\end{proof}

\section{Arithmetic regularity and structures in difference sets}\label{section:AR}
In this section we apply the uniform continuity result to prove Theorems \ref{thm:AR} and \ref{thm:ARBohrSimple} --- our VC-dimension versions of the arithmetic regularity lemma of Terry and Wolf --- as well as prove the corollaries of Bogolyubov--Ruzsa type.

We shall in fact prove the following more general result.

\begin{theorem}[Arithmetic regularity, Bohr set version]\label{thm:ARBohr}
Let $\epsilon \in (0,1)$, $\nu \in [0,1]$, and let $G$ be a finite abelian group. If $A \subset G$ has size least $\alpha \abs{G}$ and $\vcd{A} \leq d$, then there is a regular Bohr set $H \subset A-A$ of rank $m \leq C d \epsilon^{-C}\log(2/\alpha)$ and radius $c\nu\epsilon^2/m^2$, and a subset $A' \subset A$, such that $\abs{ A \symmdiff (A'+H) } \leq \epsilon \abs{A}$.\\
Moreover, we may take $\abs{A'} \geq (1-\epsilon)\abs{A}$, and there is some radius-dilate $D = Cm/\epsilon\nu$ such that $\abs{A \cap (x+H_D)} \geq (1-\epsilon)\abs{H_D}$ for all $x \in A'+H$.
\end{theorem}

\Cref{thm:AR} is an immediate consequence of this, taking $\nu = 0$, so that $H$ is a subspace of the required codimension by \Cref{lemma:BohrStructure}.

We remark that if one assumes the more general condition $\abs{A+S} \leq K \abs{A}$ instead of the density condition, then one gets the same conclusion but with the bound $m \leq C d \epsilon^{-C}\log K + C\log(2\abs{G}/\abs{S})$ on the rank instead; this comes from the bound in Theorem \ref{thm:mainAbelian}. Recall also that an improved $\epsilon$-dependence is discussed in Section \ref{section:aftermath}.

\begin{proof}[Proof of \Cref{thm:ARBohr}]
We apply Theorem \ref{thm:mainAbelian} with $B = A$ and some parameter $\delta = \delta(\epsilon) \leq 1/10$ to be specified later in place of $\epsilon$, taking $S = G$ and $K = 1/\alpha$. This gives us a regular Bohr set $T$ of rank $m \leq C d \delta^{-2}(\log 2/\delta)^2 \log(2/\alpha)$ and radius $c\epsilon /m$ such that 
\[ \abs{\mu_A * 1_{-A}(x+t) - \mu_A * 1_{-A}(x)} \leq \delta \quad\text{for all $x \in G$ and $t \in T$.} \] 
Since $\mu_A * 1_{-A}(0) = \E_{a \in A} 1_A(a) = 1$, looking at just $x = 0$ tells us that $\mu_A * 1_{-A}(t) \geq 1-\delta$ for all $t \in T$. Note in particular that $T \subset A-A$; recall also that $T=-T$. Averaging the previous inequality over all $t \in T$, we have
\[ \mu_A * 1_{-A} * \mu_T(0) \geq 1-\delta, \]
and so
\[ \E_{a \in A}\, 1_A*\mu_T(a) \geq 1-\delta. \]
Thus, on average over $a \in A$, $1_A*\mu_T(a) = \abs{A \cap (a+T)}/\abs{T} \geq 1-\delta$. Let 
\[ A' = \{ a \in A : 1_A*\mu_T(a) \geq 1 - \delta^{1/2} \} \subset A. \]
By Markov's inequality, $\mu_A(A') \geq 1-\delta^{1/2}$, since $1_A*\mu_T(a) \leq 1$ for all $a$. Let $H = T_\tau$, where $\tau = c \nu\delta^{1/2}/m$ is picked so that $H$ is regular using Lemma \ref{lemma:bohrReg}. Lemma \ref{lemma:regConv} then yields
\[ \abs{1_A*\mu_T(a+t) - 1_A*\mu_T(a)} \leq \sum_{x \in G} \abs{\mu_T(x+t) - \mu_T(x)} \leq \delta^{1/2} \]
for every $t \in H$ and every $a \in G$. Then
\begin{equation}
1_A*\mu_T(x) \geq 1-2\delta^{1/2} \quad \text{for all $x \in A'+H$.} \label{eqn:TH}
\end{equation}
Equivalently, $\left(1-2\delta^{1/2}\right) 1_{A'+H} \leq 1_A*\mu_T$, and summing this over the whole group, using the fact that $\norm{1_A*\mu_T}_1 = \norm{1_A}_1 \norm{\mu_T}_1 = \abs{A}$, yields that
\[ \left(1-2\delta^{1/2}\right) \abs{A'+H} \leq \abs{A}. \]
Since $\abs{A \cap (A'+H)} \geq \abs{A'} \geq \left(1 - \delta^{1/2}\right)\abs{A}$, we thus have
\[ \abs{A \symmdiff (A'+H)} = \abs{A} + \abs{A'+H} - 2\abs{A \cap (A'+H)} \leq 10 \delta^{1/2} \abs{A}. \]
Taking $\delta = \epsilon^2/100$ thus yields the result. Note that the claims in the ``moreover'' part of the theorem statement are satisfied, as $\mu_A(A') \geq 1-\delta^{1/2}$ and, with $D = 1/\tau$, $H_D = T$ works by \eqref{eqn:TH}.
\end{proof}

\begin{remark}\label{rmk:translateControl}
In the context of bounded exponent groups or $\F_q^n$, one may take $H \subset T$ to be a large subgroup or subspace, respectively, by Lemma \ref{lemma:BohrStructure}, and one may then replace $H_D$ with just $H$ in the statement of the theorem, so that 
\[ \abs{A \cap (x+H)} \geq (1-\epsilon)\abs{H} \quad\text{for all $x \in A'+H$.} \]
The intersection is of course empty for $x \notin A+H$; thus we have control of $A$ on most translates of $H$.
\end{remark}

\begin{remark}
One could, using known ideas, work with dense subsets of Bohr sets here instead of assuming density in the overall group, but we do not have any applications in mind.
\end{remark}

Let us also record the following small doubling version valid over finite fields.

\begin{theorem}[Arithmetic regularity, small doubling version]\label{thm:ARdoubling}
Let $\epsilon \in (0,1)$, and let $G = \F_q^n$. If $A \subset G$ has $\abs{A-A}\leq K\abs{A}$ and $\vcd{A} \leq d$, then there is subspace $H \subset A-A$ of $(A-A)$-codimension at most $C d \epsilon^{-C}\log K  + C\log q$ and a union $W$ of cosets of $H$, contained in $A+H$, such that $\abs{ A \symmdiff W } \leq \epsilon \abs{A}$.
\end{theorem}

There are several routes to get this kind of result; this specific version follows using a modelling lemma like Lemma \ref{lemma:model}, allowing one to Freiman-isomorphically embed a set $A \subset \F_q^n$ with $\abs{A-A} \leq K\abs{A}$ into a group whose size is at most $q\,C(K)\abs{A}$, at which point the tools of the density world apply. We do not detail this here, but we shall use a similar argument shortly.

\subsection*{Structures in difference sets}

We turn now to Corollaries \ref{cor:PBR} and \ref{cor:densePBR}. Again we prove a generalisation in terms of Bohr sets. Throughout this section, $G$ represents a finite abelian group.

\begin{theorem}\label{thm:PBRBohr}
If $A \subset G$ has size at least $\alpha \abs{G}$ and $\vcd{A} \leq d$, then $A-A$ contains a Bohr set of rank at most $m \leq Cd \log(2/\alpha)$ and radius at least $c/m$.
\end{theorem}
We remark again that if one assumes $\abs{A+S} \leq K\abs{A}$ instead of the density condition, then one gets the same conclusion but with rank $m \leq Cd\log{K} + C\log(2\abs{G}/\abs{S})$, coming from the application of Theorem \ref{thm:mainAbelian}. Note also that the result almost follows directly from Theorem \ref{thm:ARBohr}, taking $\epsilon = 1/2$, but with slightly worse radius. A direct proof from Theorem \ref{thm:mainAbelian} is, however, much simpler:
\begin{proof}
Applying our main theorem, Theorem \ref{thm:mainAbelian}, with $\epsilon = 1/2$, $B=A$, $S = G$ and $K=1/\alpha$, we get a Bohr set $T$ of the required rank and radius such that, for each $t \in T$ and $x \in G$,
\[ \abs{ \mu_A* 1_{-A}(x+t) - \mu_A* 1_{-A}(x) } \leq 1/2. \]
Taking $x = 0$ and using that $\mu_A * 1_{-A}(0) = 1$, we see that $\mu_A* 1_{-A}(t) \geq 1/2$ for all $t \in T$. In particular, $t \in A-A$ for all $t \in T$, and so we are done.
\end{proof}
The almost-periodicity result of course says more: it says that not only is $T$ contained in $A-A$, but every element of $T$ is \emph{well-represented} as a difference of elements in $A$.

Corollary \ref{cor:densePBR} follows immediately from the above theorem and Lemma \ref{lemma:BohrStructure}. To prove the small doubling variant, Corollary \ref{cor:PBR}, we shall use the following lemma, which is well-known in the additive combinatorics community. It is a slight generalisation of the Green--Ruzsa modelling lemma \cite[Proposition 6.1]{GrRu:2007}, and the proof follows that in \cite{GrRu:2007}.

\begin{lemma}[Freiman modelling over finite fields]\label{lemma:model}
Let $s \geq 2$. Suppose $A \subset \F_q^n$ has $\abs{A-A}\leq K\abs{A}$ or $\abs{A+A} \leq K\abs{A}$. Then $A$ is Freiman $s$-isomorphic to a subset of $G = \F_q^m$, where $\abs{G} \leq (q-1) K^{2s} \abs{A}$.
\end{lemma}

The definition of Freiman $s$-isomorphism is similar to that of $2$-isomorphism given in Definition \ref{defn:FreimanIso}: a function $\varphi : A \to B$ is said to be a Freiman $s$-isomorphism for some $s \geq 1$ if, for all $a_i \in A$, 
\[ a_1+\cdots+a_s = a_{s+1}+\cdots+a_{2s} \Longleftrightarrow \varphi(a_1)+\cdots+\varphi(a_s) = \varphi(a_{s+1}) + \cdots + \varphi(a_{2s}). \]
Note that such a function is automatically a bijection. We remark again that an $s$-isomorphism $A \to B$ extends naturally to an $s/k$-isomorphism from $A \pm A \pm \cdots \pm A \to B \pm B \pm \cdots \pm B$ whenever $k$ divides $s$, where there are $k$ copies of the sets on each side and the same number of plusses and minuses on both sides.

\begin{proof}[Proof of Lemma \ref{lemma:model}]
Let $m$ be minimal such that $A$ is Freiman $s$-isomorphic to a subset $B$ of $G = \F_q^m$. Since $s \geq 2$, we still have that $\abs{B-B} \leq K\abs{B}$ or $\abs{B+B} \leq K\abs{B}$, whichever held for $A$.

Let $X = \F_q^\times \cdot (sB-sB) = \{ rx : r \in \F_q^\times, x \in sB-sB \}$. Assume for a contradiction that $X \neq G$, and let $x \in G \setminus X$. Denoting by $\langle x \rangle$ the subspace generated by $x$, note that $\langle x \rangle \cap X = \{0\}$, since $X$ is invariant under dilation. Now let $\varphi : G \to \F_q^{m-1}$ be a linear map with $\ker(\varphi) = \langle x \rangle$. We claim this restricts to an $s$-isomorphism from $B$ to $\varphi(B)$. Indeed, by linearity
\[ \varphi(a_1)+\cdots+\varphi(a_s) = \varphi(a_{s+1}) + \cdots + \varphi(a_{2s}) \iff a_1 + \cdots + a_s - a_{s+1} - \cdots - a_{2s} \in \ker(\varphi), \]
and since $(sB-sB)\cap \ker(\varphi) = \{0\}$ by construction, we have the claim. This contradicts the minimality of $m$, however, since the composition of two $s$-isomorphisms is again an $s$-isomorphism, and so
\[ \abs{G} = \abs{X} \leq (q-1)\abs{sB-sB} \leq (q-1)K^{2s}\abs{A}, \]
the final inequality following from the Pl\"unnecke--Ruzsa--Petridis inequality \cite[Corollary 6.29]{TV} (see \cite{Pe:2012} for an elegant proof). 
\end{proof}

We can now prove Corollary \ref{cor:PBR}, which we restate for convenience.

\PBR*
\begin{proof}
By translating if necessary, we assume without loss of generality that $A$ contains $0$. By the modelling lemma, we may embed $A$ into $G = \F_q^m$, where $\abs{G} \leq q K^8 \abs{A}$, by a $4$-isomorphism $\varphi$ taking $0$ to $0$.  Being a $4$-isomorphism, $\varphi$ extends to a $2$-isomorphism $A-A \to \varphi(A)-\varphi(A)$, whence $\abs{A-A} = \abs{\varphi(A) - \varphi(A)}$. Applying Theorem \ref{thm:PBRBohr} to $\varphi(A)$ in $G$, taking $S = -\varphi(A)$ in the bounds in the remark immediately following the theorem, we get that $\varphi(A)-\varphi(A)$ contains a subspace $V$ of $(A-A)$-codimension at most $Cd\log K + C\log q$. Hence $A-A$ contains $H \coloneqq \varphi^{-1}(V)$. Since $\varphi^{-1}$ is a $2$-isomorphism $V \to H$ taking $0$ to $0$, $H$ is also a subspace, and we are done.
\end{proof}

We remark that the $\log q$ term is of course somewhat artificial, coming from how the size of $4A-4A$ relates to a power of $q$ in the modelling lemma.

\section{VC-dimension and $k$-stability}\label{section:stability}
We make here some brief remarks about the relationship between our notion of VC-dimension and the notion of $k$-stability used by Terry and Wolf in \cite{TeWo:2017}, this being defined as follows.\footnote{Terry--Wolf used this definition restricted to the abelian setting; a general version appears in \cite{CoPiTe:2017}.}

\begin{defn}
Let $A$ be a subset of a group $G$. Then $A$ is said to have the \emph{$k$-order property} if there exist $a_1,\ldots,a_k, b_1, \ldots, b_k \in G$ such that $a_i b_j \in A$ if and only if $i \leq j$. If $A$ does not have the $k$-order property, it is called \emph{$k$-stable}.
\end{defn}

We first show that $k$-stable sets have small VC-dimension.
\begin{lemma}
Let $A \subset G$ be $k$-stable. Then $\vcd{A} \leq \vcd{A,G} \leq k-1$.
\end{lemma}
\begin{proof}
We show that if $\vcd{A,G} \geq k$, then $A$ has the $k$-order property. Let $X \subset G$ be a set of size $k$ shattered by the family $\{ xA : x \in G \}$. Writing $b_1, \ldots, b_k$ for the elements of $X$, let $a_1,\ldots,a_k \in G$ be elements such that
\[ (a_i^{-1}A) \cap X = \{ b_i,\ldots,b_k \}; \]
such elements exist by shattering. Then $a_i b_j \in A$ iff $b_j \in (a_i^{-1}A)\cap X$, which is true iff $i \leq j$ by definition of $a_i$. Thus the $k$-order property holds.
\end{proof}

In the other direction, no meaningful bound exists in general:

\begin{lemma}For every $k \geq 2$, there is a set $A$ with $\vcd{A} \leq 2$ that is not $k$-stable.
\end{lemma}
\begin{proof}
The arithmetic progression $A \coloneqq [0,k) \subset \Z$ has the $k$-order property, taking $a_i = i-1$ and $b_j = k-j$ for $i,j=1,\ldots,k$. On the other hand, by Proposition \ref{prop:APs}, $\vcd{A} = 2$ for $k \geq 3$, and equals $1$ if $k = 2$.
\end{proof}

When it comes to non-abelian groups, Conant, Pillay and Terry \cite{CoPiTe:2017} recently proved the following structurally strong result about stable sets.

\begin{theorem}[Conant--Pillay--Terry]\label{thm:CPT}
For any $k \geq 1$ and $\epsilon > 0$, there are $n = n(k,\epsilon)$ and $N = N(k,\epsilon)$ such that the following holds. Suppose $G$ is a finite group of size at least $N$, and $A \subset G$ is $k$-stable. Then there is a normal subgroup $H \leq G$, of index at most $n$, such that for each coset $C$ of $H$ either $\abs{A \cap C} \leq \epsilon \abs{H}$ or $\abs{A \cap C} \geq (1-\epsilon)\abs{H}$. Moreover, there is a union $W$ of cosets of $H$ such that $\abs{A \symmdiff W} \leq \epsilon \abs{H}$.
\end{theorem}

This is another indication of a fundamental difference between $k$-stability and bounded VC-dimension: in groups of prime order, there are sets with bounded VC-dimension of any possible size, whereas the above theorem shows that $k$-stable sets in such groups are necessarily either very small or very large.

\begin{remark}\label{rmk:translateControlCPT}
In the case of stability, both the results of Terry--Wolf \cite{TeWo:2017} and Conant--Pillay--Terry \cite{CoPiTe:2017} show that for each $t \in A+H$, $\abs{A \cap (t+H)}$ is either very large or very small. (For $t \notin A+H$, the intersection is empty.) In the VC-bounded case, Theorem \ref{thm:ARBohr} shows that there is a Bohr set $H$, a dilate $H'$ thereof and an almost-full subset $A' \subset A$ for which the intersections $A\cap(t+H')$ are large for all $t \in A'+H$. For $t \in (A+H')\setminus (A'+H)$, however, it says nothing; the case of arithmetic progressions indicates that some of the intersections can be of medium size --- thus one cannot in general hope to obtain the same kind of dichotomy.
\end{remark}

\begin{remark}
As noted in the introduction, around the same time as this work, Conant--Pillay--Terry \cite{CoPiTe:2018} proved a version of Theorem \ref{thm:CPT} for sets with bounded VC-dimension that shows what the regularity conclusion can look like for general (not necessarily abelian) groups; see \cite[Theorem 5.7]{CoPiTe:2018}. A subsequent paper of Conant--Pillay \cite{CoPi:2020} reaches even further, assuming only a small tripling condition rather than a density condition; see \cite[Theorem 2.1]{CoPi:2020}. Let us also note that in the abelian setting, using in part methods from the current paper, Terry--Wolf \cite{TeWo:2018} have proved a strong regularity lemma for stable sets in abelian groups; see \cite[Theorem 3]{TeWo:2018}.
\end{remark}

\section{Improving the $\epsilon$-dependence}\label{section:aftermath}
Around the same time as our work, Alon, Fox and Zhao \cite{AFZ:2018} proved using different methods an arithmetic regularity lemma for sets $A$ where the family $\{ t + A : t \in G \}$ has small VC-dimension, for abelian groups with bounded exponent. In our notation, this notion of VC-dimension corresponds precisely to $\vcd{A,G}$. This notion was also looked at earlier by P. Simon \cite{Sim:2017} in the context of locally compact groups, with motivation coming from model theory. An example of where this notion is different to our definition of $\vcd{A}$ is for $A$ being a coset of a proper subgroup of $G$: then $\vcd{A,G} = 1$, whereas $\vcd{A} = 0$, by Proposition \ref{prop:vcd0}. In fact, differing by $1$ is as bad as the difference can get, by Proposition \ref{prop:monotonicity}. We remark that in general $\vcd{A,B}$ can be very different to $\vcd{A}$, however.

As mentioned in the introduction, the approach of Alon--Fox--Zhao gives a superior $\epsilon$-dependence in the arithmetic regularity lemma for bounded exponent groups. Using a variant of their very interesting lemma \cite[Lemma 2.2]{AFZ:2018}, which is based on Haussler's discrete sphere packing lemma \cite{Haussler}, one can in fact improve the $\epsilon$-dependence from polynomial to logarithmic in Theorem \ref{thm:mainAbelian}, at least in the case $B=A$. The result one obtains is the following:

\begin{theorem}\label{thm:HausslerCty}
Let $\epsilon \in (0,1]$. Let $G$ be a finite abelian group and suppose $A \subset G$ is a subset with $\abs{A} \geq \alpha \abs{G}$ and $\vcd{A,G} \leq d$. Then there is a regular Bohr set $T$ of rank $m \leq C d \log(2/\epsilon\alpha)$ and radius at least $c\epsilon/m$ such that, for each $t \in T$,
\[ \norm{\mu_A* 1_{-A}(\cdot + t) - \mu_A* 1_{-A}}_\infty \leq \epsilon. \]
\end{theorem}

In particular, this means that the codimension and rank bounds in Theorems \ref{thm:AR}, \ref{thm:ARBohrSimple} and \ref{thm:ARBohr} can all be improved to $Cd\log(2/\epsilon\alpha)$, and in Theorem \ref{thm:ARdoubling}, the codimension can be taken to be $C d \log(2K/\epsilon)  + C\log q$ (by a slight generalisation of the above).

We remark that one can improve this even further in the case of groups of bounded exponent, but \cite{AFZ:2018} is already very efficient in the bounded exponent setting. The main advantage of the methods of this paper is that they apply equally strongly to the `characteristic 0' setting.

As in the proof of Theorem \ref{thm:mainAbelian}, we prove the above result using the following analogue of Theorem \ref{thm:mainSum}.

\begin{theorem}\label{thm:mainSumHaussler}
Let $\epsilon \in (0,1]$ and $d, k \in \N$. Let $G$ be a group and let $B \subset G$ be a finite subset with $\vcd{B,G} \leq d$. If $\abs{S \cdot B} \leq K\abs{B}$ for some set $S \subset G$, then there is a set $T \subset S$ of size at least $(c\epsilon/Kk)^d \abs{S}$ such that, for any set $A \subset G$ and each $t \in (T^{-1} T)^k$,
\[ \norm{ \tau_t(\mu_B* 1_{A^{-1}}) - \mu_B* 1_{A^{-1}}}_\infty \leq \epsilon. \]
\end{theorem}

This has a better $\epsilon/k$-dependence than Theorem \ref{thm:mainSum}, but note that the hypotheses are different: the VC-dimension bound is placed on $B$ here, not $A$, and is relative to the whole group rather than a subset. Thus, in this version, both the VC-bound and the small sumset condition are placed on the same set, whereas in Theorem \ref{thm:mainSum} they need not be coupled. When $B=A$, the case used for the proof of the arithmetic regularity lemmas, there is of course no distinction. The theorem actually follows immediately from the following $L^1$-statement, which is a very slight variant of Alon--Fox--Zhao's \cite[Lemma 2.2]{AFZ:2018}:

\begin{theorem}\label{thm:mainSumHausslerL1}
Let $\epsilon \in (0,1]$ and $d, k \in \N$. Let $G$ be a group and let $B \subset G$ be a finite subset with $\vcd{B,G} \leq d$. If $\abs{S \cdot B} \leq K\abs{B}$ for some set $S \subset G$, then there is a set $T \subset S$ of size at least $(c\epsilon/Kk)^d \abs{S}$ such that, for each $t \in (T^{-1} T)^k$,
\[ \norm{ \tau_t(\mu_B) - \mu_B}_1 \leq \epsilon. \]
\end{theorem}

In order to prove this, following \cite{AFZ:2018}, we require the following definition and lemma of Haussler. We say that a family $\mathcal{F}$ of sets on a ground set $X$ is $\delta$-separated if $\abs{U \symmdiff V} \geq \delta \abs{X}$ for any distinct $U, V \in \mathcal{F}$. The aforementioned sphere packing lemma \cite[Theorem 1]{Haussler} of Haussler implies the following bound on such families with bounded VC-dimension. 

\begin{lemma}\label{lemma:Haussler}
Let $\delta > 0$ and let $\mathcal{F}$ be a family of $\delta$-separated sets on a ground set $X$, with $\vcd{\mathcal{F}} = d$. Then $\abs{\mathcal{F}} \leq (C/\delta)^d$. 
\end{lemma}

\begin{proof}[Proof of Theorem \ref{thm:mainSumHausslerL1}]
The argument is a slight variant of the proof of \cite[Lemma 2.2]{AFZ:2018}. Let $\mathcal{B} = \{ xB : x \in S\}$ be a family of sets on the ground set $S\cdot B$, and note that $\vcd{\mathcal{B}} \leq \vcd{B,G} \leq d$. Let $\delta > 0$ be a parameter to be specified later, and let $U \subset S$ be maximal such that the collection $\{ xB : x \in U \}$ is $\delta$-separated. By Lemma \ref{lemma:Haussler}, $\abs{U} \leq (C/\delta)^d$. By maximality, for any $y \in S$, there is some $u \in U$ such that $\abs{yB \symmdiff uB} \leq \delta \abs{S\cdot B}$. Writing
\[ T_0 = \{ t \in G : \abs{tB \symmdiff B} \leq \delta \abs{S\cdot B} \}, \]
this means by translation-invariance that for any $y \in S$, there is some $u \in U$ such that $y \in uT_0$. Thus $S \subset \cup_{u \in U} (uT_0 \cap S)$, and so there is some $u \in U$ for which the set $T \coloneqq u T_0 \cap S$ has size
\[ \abs{T} \geq (c\delta)^d \abs{S}. \]
It remains to show that for an appropriate choice of $\delta$, this set $T$ satisfies the conclusion of the theorem. For this, note that if $t \in T^{-1}T$, then we can write $t = (u t_1)^{-1}(u t_2) = t_1^{-1}t_2$ for some $t_i \in T_0$. Thus
\begin{align*}
\norm{ \tau_t(1_B) - 1_B }_1 &= \abs{t^{-1} B \symmdiff B } \\
&= \abs{ t_1 B \symmdiff t_2 B} \\
&\leq \abs{t_1 B \symmdiff B} + \abs{t_2 B \symmdiff B} \\
&\leq 2\delta\abs{S \cdot B}.
\end{align*}
We pick $\delta = \epsilon/2Kk$, so that this is at most $(\epsilon/k)\abs{B}$. By the triangle inequality, we then have that for any $t \in (T^{-1}T)^k$, 
\[ \norm{ \tau_t(1_B) - 1_B}_1 \leq \epsilon\abs{B}, \]
as required.
\end{proof}

\begin{proof}[Proof of Theorem \ref{thm:mainSumHaussler}]
This follows immediately from Theorem \ref{thm:mainSumHausslerL1}, as, for any $x \in G$ and $t \in (T^{-1}T)^k$,
\begin{align*}
\abs{ \mu_B* 1_{A^{-1}}(tx) - \mu_B* 1_{A^{-1}}(x) } &= \abs{ \mu_{t^{-1}B}*1_{A^{-1}}(x) - \mu_B*1_{A^{-1}}(x) }, \\
&= \abs{ \sum_{y \in G} (\mu_{t^{-1}B}(y) - \mu_B(y))1_{A^{-1}}(y^{-1}x) } \\
&\leq \norm{ \tau_t(\mu_{B}) - \mu_B }_1.\qedhere
\end{align*}
\end{proof}

\begin{proof}[Proof Theorem \ref{thm:HausslerCty}]
The proof follows the proof of Theorem \ref{thm:mainAbelian} exactly, only substituting Theorem \ref{thm:mainSumHaussler} in place of Theorem \ref{thm:mainSum}. 
\end{proof}

\section{Conclusion}\label{section:conclusion}

\subsection*{Other notions of dimension}
Most of the results of this paper would be valid under assumptions more general than VC-dimension being small, such as low primal shatter dimension, low metric entropy, or low Rademacher complexity. Indeed, versions of Theorem \ref{thm:uniformEmp} hold with such assumptions in place of bounded VC-dimension, and this was the only place VC-dimension was used in the proofs of the main results (prior to Section \ref{section:aftermath}). VC-dimension seems to be the most widely studied notion, however, so we have chosen to phrase the results in terms of this. In a similar vein, it would also be natural to consider functions more general than indicator functions in the continuity results, using the flexibility of these more general assumptions.

\subsection*{The locally compact setting}
Theorem \ref{thm:mainSum} should generalise readily to the setting of second countable locally compact groups. The only required changes are to replace each reference to cardinality by (left) Haar measure instead, and to add measurability conditions.

\subsection*{Questions}
Much remains to be investigated surrounding this notion of VC-dimension. For example, although we know that Bohr sets have bounded VC-dimension, and we obtain our characterisation in terms of approximations by Bohr sets, it is not clear that one cannot prove a much stronger classification if one allows other types of sets. Determining an appropriate family that is both necessary and sufficient would be interesting.

On the additive side, the following seem to be natural questions.

What is the likely VC-dimension of a random subset of $G$ of size $m$?

If $A \subset \F_2^n$ has density at least $0.49$, must there be a dense subset $B \subset A$ with $\vcd{B} = o(n)$? (One might want to impose some conditions on $A+A$ here.)

If $A \subset \F_2^n$ has density at least $\alpha$, must $2A-2A$ contain a subset of bounded VC-dimension and size at least $\alpha^{O(1)}\abs{A}$?

The following question arose with Sean Prendiville: given $d$, what is the largest size of a 3AP-free subset $A$ of a finite abelian group $G$ (or $[N]$) if $\vcd{A} \leq d$? Being 3AP-free here means that the set does not contain any non-trivial three-term arithmetic progressions, that is solutions to $x+z=2y$ with $x,y,z$ distinct elements of $A$. Chapman \cite{Chapman} also observed Proposition \ref{prop:Sidon} and noted that Sidon sets thus provide examples of somewhat large 3AP-free sets for $d=2$. Is $O(N^{1/2})$ the correct answer for $d=2$? For larger $d$, it seems likely that the Behrend construction can be used to provide larger examples; whether this is close to optimal or not is completely unclear.

\section*{Acknowledgments}
This work was partly carried out at the Pseudorandomness programme at the Simons Institute for the Theory of Computing, whose support is gratefully acknowledged. The author would like to thank Thomas Bloom and Caroline Terry for helpful conversations about VC-dimension, and the referee for numerous suggestions that led to improvements to the paper, and in particular for supplying the simplified proof of the modelling lemma, Lemma \ref{lemma:model}.

\bibliographystyle{amsplain}


\begin{dajauthors}
\begin{authorinfo}[olof]
  Olof Sisask\\
  Department of Mathematics\\
  Stockholm University\\
  SE-106 91 Stockholm\\
  Sweden\\
  olof.sisask\imageat{}math\imagedot{}su\imagedot{}se \\
  \url{https://www.sisask.com}
\end{authorinfo}
\end{dajauthors}

\end{document}